\documentclass[a4paper,11pt]{article}

\usepackage{amsmath}
\usepackage{amssymb}
\usepackage{amsthm}
\usepackage{graphicx}
\usepackage{parskip}
\usepackage{fancyhdr}
\usepackage{dsfont}
\usepackage{fullpage}
\usepackage{hyperref}

\newtheorem{thm}{Theorem}
\newtheorem{prop}{Proposition}

\newtheorem{cor}{Corollary}
\theoremstyle{definition}
\newtheorem{defn}{Definition}
\newtheorem{remark}{Remark}

\begin{document}

\title{Computational inference beyond Kingman's coalescent} 

\author{Jere Koskela \\
	\texttt{j.j.koskela@warwick.ac.uk}\\
	\small Mathematics Institute \\
	\small University of Warwick \\
	\small Coventry CV4 7AL \\
	\small UK
	\and
	Paul A. Jenkins \\
	\texttt{p.jenkins@warwick.ac.uk}\\
	\small Department of Statistics \\
	\small University of Warwick \\
	\small Coventry CV4 7AL \\
	\small UK
	\and
	Dario Span\`{o} \\
	\texttt{d.spano@warwick.ac.uk}\\
	\small Department of Statistics \\
	\small University of Warwick \\
	\small Coventry CV4 7AL \\
	\small UK
} 

\date{\today}

\maketitle

\begin{abstract}
Full likelihood inference under Kingman's coalescent is a computationally challenging problem to which importance sampling (IS) and the product of approximate conditionals (PAC) method have been applied successfully.
Both methods can be expressed in terms of families of intractable conditional sampling distributions (CSDs), and rely on principled approximations for accurate inference.
Recently, more general $\Lambda$- and $\Xi$-coalescents have been observed to provide better modelling fits to some genetic data sets.
We derive families of approximate CSDs for finite sites $\Lambda$- and $\Xi$-coalescents, and use them to obtain ``approximately optimal" IS and PAC algorithms for $\Lambda$-coalescents, yielding substantial gains in efficiency over existing methods.
\end{abstract}



\section{Introduction}\label{intro}

Importance sampling (IS) has a well established role in population genetic inference as a means of approximating likelihoods.
In this context the method was introduced by Griffiths and Tavar\'{e}, who derived a recursion for quantities of interest under Kingman's coalescent \cite{Kingman82} and simulated a Markov chain to approximate its solution \cite{Griffiths94a}, \cite{Griffiths94b}, \cite{Griffiths94c}, \cite{Griffiths99}.
Their approach was identified as importance sampling by Felsenstein et al. \cite{Felsenstein99}, which led Stephens and Donnelly to derive the optimal proposal distribution in terms of a family of conditional sampling distributions (CSDs) \cite{Stephens00}.
The CSDs are inaccessible in general, but the authors introduced an approximation which yielded dramatic improvements in efficiency and accuracy of IS algorithms.
In addition, IS has been investigated and applied to genetic problems such as demographic and other parameter inferences in \cite{Griffiths96}, \cite{Fearnhead01}, \cite{DeIorio05}, \cite{Griffiths08}, \cite{Gorur08}, \cite{Hobolth08} and \cite{Jenkins11}.

Approximating the CSDs for various generalisations of Kingman's coalescent has received plenty of attention, both as a means of deriving approximations to the optimal importance sampling algorithm and due to the product of approximate conditionals (PAC) method introduced in \cite{Li03}.
De Iorio and Griffiths derived an approximation to finite alleles CSDs based on the Fleming-Viot generator \cite{DeIorio04a}, \cite{DeIorio04b} and Paul and Song provided a genealogical interpretation and included crossover recombination \cite{Paul10}.
Further approximations based on Hidden Markov Models have been obtained in \cite{Paul11} and \cite{Steinrucken13b}, and applied in \cite{Sheehan13}.

Kingman's coalescent only permits binary mergers of ancestral lineages.
The $\Lambda$-coalescents, introduced by Pitman \cite{Pitman99} and Sagitov \cite{Sagitov99}, generalise Kingman's coalescent by permitting multiple lineages to merge in one event.
The rate of coalescence of any $k$ out of $n$ lineages is given by
\begin{equation*}
\lambda_{n,k} := \int_0^1 r^k ( 1 - r )^{ n - k } \frac{ 1 }{ r^2 } \Lambda( dr )
\end{equation*}
for some finite measure $\Lambda$ on $[0,1]$, which can be taken to be a probability measure without loss of generality.
Popular choices of $\Lambda$ include $\Lambda = \delta_0$, which corresponds to Kingman's coalescent, $\Lambda = \delta_1$ leading to star-shaped genealogies, $\Lambda = \frac{ 2 }{ 2+ \psi^2 } \delta_0 + \frac{ \psi^2 }{ 2 + \psi^2 } \delta_{ \psi }$ where $\psi \in ( 0, 1 ]$ \cite{Eldon06} and $\Lambda = \operatorname{Beta}( 2 - \alpha, \alpha )$ where $\alpha \in ( 1, 2 )$ \cite{Schweinsberg03}.
See \cite{Birkner09b} for a review.

Investigations by  Boom et al. \cite{Boom94},  \'{A}rnason \cite{Arnason04}, Eldon and Wakeley \cite{Eldon06}, and Birkner and Blath \cite{Birkner08} have concluded that $\Lambda$-coalescents can provide better descriptions of some populations than Kingman's coalescent, particularly among marine species.
Thus, similar strategies of inference have been developed for them.
An analogue of the Griffiths-Tavar\'{e} recursion (see equation \eqref{sequential_likelihood} in Section \ref{optimal_proposal}) for $\Lambda$-coalescents was derived by Birkner and Blath in \cite{Birkner08}.
In a subsequent paper \cite{Birkner11} the authors characterised the optimal IS proposal distribution in terms of a family of Green's functions related to the time-reversal of the $\Lambda$-coalescent, and used their representation to obtain an approximately optimal algorithm for the infinite sites model of mutation.
\cite{Steinrucken13a} contains a detailed discussion of inference under Beta-coalescents and their applicability to marine populations.

The $\Lambda$-coalescent family allows for multiple mergers, but only permits one merger at a time.
They are generalised further by the $\Xi$-coalescents, which permit any number of simultaneous, multiple mergers.
$\Xi$-coalescents were introduced by M\"{o}hle and Sagitov \cite{Moehle01} and Schweinsberg \cite{Schweinsberg00} and can be expressed in terms of a finite measure $\Xi$ (again,  a probability measure without loss of generality) on the infinite simplex
\begin{equation*}
\Delta = \left\{ \mathbf{ r } = ( r_1, r_2, \ldots ) \in [ 0, 1 ]^{ \mathbb{ N } } : \sum_{ i = 1 }^{ \infty } r_i \leq 1 \right\}.
\end{equation*}
We denote by $\lambda_{ n; k_1, \ldots, k_p; s }$ the rate of jumps involving $p \geq 1$ mergers with sizes $k_1, \ldots, k_p$, with $s = n - \sum_{ i = 1 }^p k_i$ lineages not participating in any merger. 
The total number of lineages before the mergers is denoted by $n$.
This rate is given as
\begin{equation*}
\lambda_{ n; k_1, \ldots, k_p; s } = \int_{ \Delta } \sum_{ l = 0 }^s \binom{s}{l} \sum_{ i_1 \in \mathbb{ N } } \ldots \sum_{ i_{ p + l } \in \mathbb{ N } } r_{ i_1 }^{ k_1 } \ldots r_{ i_p }^{ k_p } r_{ i_{ p+1 } } \ldots r_{ i_{ p + l } } \frac{ \left( 1 - \sum_{ i = 1 }^{ \infty } r_i \right)^{ s - l } }{ \sum_{ i = 1 }^{ \infty } r_i^2 } \Xi( d\mathbf{ r } ).
\end{equation*}
$\Xi$-coalescents have also been used to model genealogies of marine organisms \cite{Sargsyan08} and populations undergoing mass extinctions \cite{Taylor09}, although the question of which measures $\Xi$ are biologically relevant remains open.
Note that if $\Xi$ assigns full mass to the set $\{ \mathbf{ r } \in \Delta : r_2 = r_3 = \ldots = 0 \}$ the resulting process is a $\Lambda$-coalescent.

In this paper we characterise the optimal IS proposal distribution for finite sites, finite alleles $\Lambda$- and $\Xi$-coalescents in terms of respective families of CSDs, and derive principled approximations to the CSDs for both coalescent families.
The rest of the paper is laid out as follows.
In Section \ref{optimal_proposal} we give a heuristic description of $\Lambda$-coalescents and derive their optimal proposal distributions.
In Section \ref{csds} we give principled derivations for approximate $\Lambda$-coalescent CSDs, based on the finite alleles $\Lambda$-Fleming-Viot generator and on genealogical considerations.
Section \ref{simulation} presents simulation studies using both IS and PAC algorithms on data sets simulated from a finite sites $\Lambda$-coalescent.
In Section \ref{xi_coalescents} we generalise the IS algorithm for the $\Xi$-coalescent family, and derive the analogous approximate CSDs.
Section \ref{discussion} concludes with a discussion.

\section{The \texorpdfstring{$\Lambda$}{Lambda}-coalescent and its optimal proposal distributions}\label{optimal_proposal}

In the notation of \cite{Paul10}, let $L = \{ 1, \ldots, | L | \}$ be a set of loci, $E_l$ be the finite set of alleles at locus $l \in L$, $\theta_l$ be the mutation rate at locus $l \in L$ and $P^{ ( l ) }$ be a family of stochastic matrices giving transition probabilities of mutations at locus $l \in L$.
Let $\theta := \sum_{ l \in L } \theta_l$ be the total mutation rate, $\mathcal{ H } := E_1 \times \ldots \times E_{ | L | }$ be the set of haplotypes and
\begin{equation*}
\Delta_{ \mathcal{ H } } = \left\{ \mathbf{ x } = ( x_1, \ldots, x_{ | \mathcal{ H } | } ) \in [ 0, 1 ]^{ | \mathcal{ H } | } : \sum_{ h \in \mathcal{ H } } x_h = 1 \right\}
\end{equation*}
be the space of probability vectors of allele frequencies. 
Denote a sample by $\mathbf{ n } = ( n_h )_{ h \in \mathcal{ H } } \in \mathbb{ N }^{ | \mathcal{ H } | }$ and let $n := \sum_{ h \in \mathcal{ H } } n_h$.
Let $\mathbf{ e }_h$ be the canonical unit vector with a 1 in position $h$ and zeros elsewhere.
For $l \in L$ and $h \in \mathcal{ H }$ let $h[ l ]$ denote the allele at locus $l$ of haplotype $h$.
Finally, for $a \in E_l$ let $S_l^a( h )$ be the haplotype obtained from $h$ by overwriting locus $l$ by allele $a$.

The dynamics of a finite sample of individuals from the stationary $\Lambda$-Fleming-Viot process under the finite alleles model of mutation can be described as follows.
For a rigorous account see \cite{Donnelly99}:

Consider a sample of $n$ typed lineages and associate to each lineage a unique level from $\{ 1, \ldots , n \}$.
Let $\Pi$ be a Poisson process on $\mathbb{ R }_+ \times [ 0, 1 ]$ with rate $dt \otimes r^{-2} \Lambda( dr )$.
At each $( t, r ) \in \Pi$ every lineage flips a coin with success probability $r$, and all successful lineages ``look down" and copy the type of the participating lineage with the lowest level.
Independently, the type of each lineage mutates at locus $l$ with rate $\theta_l$ and jumps drawn from $P^{ ( l ) }$.
This particle system embeds the $\Lambda$-coalescent, with coalescence events traced along the look-down-and-copy jumps.

Let $( H_i )_{ i = 0 }^{- T }$ denote the sequence of type configurations of the ancestral sample after $i$ events, whether they be mutations or coalescences.
$H_0$ contains the frequency counts of the observed sample, $H_{ -T }$ is the type of the most recent common ancestor (MRCA) and the other $H_i$'s correspond to intermediate states along the ancestral tree.
Note that the likelihood $\mathbb{ P }( H_0 )$ can be written as
\begin{equation*}
\mathbb{ P }( H_0 ) = \sum_{ \mathcal{ A } } \mathbb{ P }( H_0 | A ) \mathbb{ P }( A ),
\end{equation*}
where $\mathcal{ A }$ is the space of possible ancestries and $\mathbb{ P }( H_0 | A ) = 1$ if the leaves of $A$ are compatible with $H_0$, and 0 otherwise.
These ancestries can be decomposed into a sequence of updates as above to give
\begin{equation}\label{sequential_likelihood}
\mathbb{ P }( H_0 ) = \sum_{ H_0 } \sum_{H_{-1}} \ldots \sum_{H_{-T}} \mathbb{ P }( H_0 | A ) \prod_{ i = -1 }^{ -T }\mathbb{ P }( H_{ i + 1 } | H_i ) \mathbb{ P }( H_{ -T } ) 
\end{equation}
where $\mathbb{ P }( H_{ -T } )$ is the invariant distribution of the mutation operator obtained by viewing mutation as a mixture Markov chain on $\mathcal{ H }$ with weights $\{ \theta_l / \theta \}_{l \in L}$ and mixture components $\{ P^{(l)} \}_{l \in L}$, and
\begin{equation}\label{forwards_probs}
\mathbb{ P }( H_{ i + 1 } | H_i ) =
\begin{cases}
\frac{ \theta_l }{ n \theta - q_{ n_i n_i } }( ( n_i )_{ S_l^a( h ) } + 1 - \delta_{ a h[ l ] } ) P^{ ( l ) }_{ a h[ l ] }  &\text{ if } H_{ i + 1 } = H_i - \mathbf{ e }_{ S_l^a( h ) } + \mathbf{ e }_h \\
\binom{ n_i }{ k } \frac{ \lambda_{ n_i,  k } }{ n \theta - q_{ n_i n_i } } \frac{ ( n_i )_h - k + 1 }{ n_i - k + 1 } &\text{ if } H_{ i + 1 } = H_i + ( k - 1 ) \mathbf{ e }_h
\end{cases} 
\end{equation}
where $- q_{ n n } = \sum_{ j = 1 }^{ n - 1 } \binom{ n }{ n - j + 1 } \lambda_{ n, n - j + 1 }$ is the total rate of coalescence of $n$ untyped lineages, $n_i$ is the number of lineages in $H_i$ and $( n_i )_h$ is the number of lineages of type $h$ in $H_i$. See \cite{Birkner08} for a detailed derivation.

As with Kingman's coalescent, \eqref{sequential_likelihood} can be approximated by sampling $N$ independent ancestors from the stationary distribution of the mutation mechanism, generating an ancestral tree $A$ from each ancestor and counting the proportion of trees with leaves that are compatible with $H_0$.
However, obtaining a nonzero estimator with reasonable probability requires a prohibitively large number of simulations as likelihoods of $\mathcal{ O }( 10^{ -10 } )$, or much smaller still, are typical even among modest data sets.
A better approach is to start with the data, propose mutations and coalescences backwards in time until the MRCA is reached, and thus ensure every simulated tree is compatible with the observed leaves.
This yields
\begin{align}
\mathbb{ P }( H_0 ) &= \sum_{H_{-T}} \ldots \sum_{H_{-1}} \prod_{ i = -1 }^{ -T }\frac{ \mathbb{ P }( H_{ i + 1 } | H_i ) }{ \mathbb{ Q }( H_i | H_{ i +1 } ) } \mathbb{ P }( H_{ -T } ) \mathbb{ Q }( H_i | H_{ i + 1 } ) \nonumber \\
&= \mathbb{ E }\left[ \prod_{ i = -1 }^{ -T } \frac{ \mathbb{ P }( H_{ i + 1 }  | H_ i ) }{ \mathbb{ Q }( H_i | H_{ i + 1 } ) } \mathbb{ P }( H_{ - T } )\right] \label{backwards_proposal}
\end{align}
where $\mathbb{ Q }( \cdot | H_{ i + 1 } )$ is an arbitrary proposal distribution satisfying mild support conditions, and the expectation is with respect to $\otimes_{ i = - 1 }^{ - T } \mathbb{ Q }( H_i | H_{ i + 1 } )$.
The expectation in \eqref{backwards_proposal} can be approximated by the IS estimator
\begin{equation*}
\widehat{ p }( H_0 ) = \frac{ 1 }{ N }\sum_{ j = 1 }^N \prod_{ i = -1 }^{ - T_j } \frac{ \mathbb{ P }( H_{ i + 1 }^{ ( j ) } | H_i^{ ( j ) } ) }{ \mathbb{ Q }( H_i^{ ( j ) } | H_{ i + 1 }^{ ( j ) } ) } \mathbb{ P }( H_{ - T_j }^{ ( j ) } ), 
\end{equation*}
where $\left\{ \{ H_i^{ ( j ) } \}_{ i = 0 }^{ - T_j } \right\}_{j = 1}^N$ is an i.i.d.~sample of sequentially constructed coalescent trees from the distributions $\mathbb{ Q }( H_i | H_{ i + 1 } )$.

The following theorem is a $\Lambda$-coalescent analogue of Theorem 1 of \cite{Stephens00}.
A similar result, using ratios of Greens functions instead of CSDs, is presented in Lemma 2.2 of \cite{Birkner11}.
\vskip 11pt
\begin{thm}\label{lambda_optimal_proposal}
Let $\pi( \mathbf{ m } | \mathbf{ n } )$ denote the sampling distribution of the next $m$ individuals given the types of the first $n$ from a population evolving according to the stationary $\Lambda$-Fleming-Viot process.
Then the optimal proposal distributions $\mathbb{ Q  }^*$ are given by
\begin{equation*}
\mathbb{ Q }^*( H_i | H_{ i + 1 } ) \propto \begin{cases}
( n_{ i + 1 } )_h \theta_l \frac{ \pi( \mathbf{ e }_{ S_l^a( h ) } | H_{ i + 1 } - \mathbf{ e }_h ) }{ \pi( \mathbf{ e }_h | H_{ i + 1 } - \mathbf{ e }_h ) } P^{ ( l ) }_{ a h[ l ] } &\text{if } H_i = H_{ i + 1 } - \mathbf{ e }_h + \mathbf{ e }_{ S_l^a( h ) } \\
\frac{ \binom{ ( n_{ i + 1 } )_{ h } }{ k } \lambda_{ n_{ i + 1 }, k } }{ \pi( ( k -1 ) \mathbf{ e }_h | H_{ i + 1 } - ( k - 1 ) \mathbf{ e }_h ) } &\text{if } H_i = H_{ i + 1 } - ( k - 1 ) \mathbf{ e }_h
\end{cases}
\end{equation*}
where the first term ranges over all possible mutations for all haplotypes present in the sample, and the second over all present haplotypes and $k \in \{ 2, \ldots ( n_{ i + 1 } )_h \}$.
\end{thm}
\begin{proof}
The argument giving the mutation term is identical to that in Theorem 1 of \cite{Stephens00} and is omitted.

For the coalescence term suppose the $n$ lineages evolve according to the lookdown construction of \cite{Donnelly99}, and denote the types of the $n$ particles at time $t$ by $D_n( t ) = ( h_1, \ldots, h_n )$.
Define $\Upsilon_k$ as the event that  in the last $\delta$ units of time there was a merger involving lineages $n - k + 1, n - k + 2, \ldots, n$.
To simplify the presentation let $h_{ i : j } := ( h_i, h_{ i + 1 }, \ldots, h_{ j - 1 }, h_j )$.
Then
\begin{align*}
\mathbb{ P }&( \Upsilon_k | D_n( t ) = ( h_{ 1 : n - k }, h, \ldots, h ) ) \\
&= \displaystyle \sum_{ g_{ 2 : k } \in \mathcal{ H }^{ k - 1 } } \frac{ \mathbb{ P }( \Upsilon_k \cap D_n( t - \delta ) = ( h_{ 1 : n - k }, h, g_{ 2 : k } ) \cap D_n( t ) = ( h_{ 1 : n - k }, h, \ldots, h ) ) }{ \mathbb{ P }( D_n( t ) = ( h_{ 1 : n - k }, h, \ldots, h )  ) } \\
&= \sum_{ g_{ 2 : k } \in \mathcal{ H }^{ k - 1 } } \frac{ \mathbb{ P }( D_n( t - \delta ) = ( h_{ 1 : n - k }, h, g_{ 2 : k } ) ) \delta \lambda_{ n, k }  }{ \mathbb{ P }( D_n( t ) = ( h_{ 1 : n - k }, h, \ldots, h )  ) } + o( \delta ) \\ 
&= \frac{ \delta \lambda_{ n, k } }{ \pi( ( k - 1 ) \mathbf{ e }_h | D_n( t ) - ( k - 1 ) \mathbf{ e }_h ) } + o( \delta ).
\end{align*}
By exchangeability every set of $k$ lineages coalesces at this same rate, so the total rate is obtained by multiplying by $\binom{ n_h }{ k }$.
\end{proof}
\begin{remark}\label{univariate_remark}
It is tempting to simplify the situation further by decomposing
\begin{equation*}
\pi( ( k - 1 ) \mathbf{ e }_h | \mathbf{ n } - ( k - 1 ) \mathbf{ e }_h ) = \prod_{ j = 0 }^{ k - 2 } \pi( \mathbf{ e }_h | \mathbf{ n } - ( k - 1 + j ) \mathbf{ e }_h )
\end{equation*}
and thus requiring only univariate CSDs.
In general a decomposition like this requires exchangeability, which the CSDs satisfy but typically approximations do not.
However, in the $\Lambda$-coalescent setting the argument being decomposed will always consist of only one type of allele.
Permuting lineages which feature only in the sample being conditioned on does not affect the outcome even for non-exchangeable families of distributions, so in this particular context univariate CSDs are sufficient.
Note that this will not be the case for $\Xi$-coalescents since simultaneous mergers of several types of lineages is permitted.
\end{remark}

\section{Approximating the \texorpdfstring{$\Lambda$}{Lambda}-coalescent CSDs}\label{csds}

An approximation to the CSDs for Kingman's coalescent was derived in \cite{DeIorio04a} by noting that the Fleming-Viot generator can be written component-wise as $\mathcal{ L } = \sum_{ h \in \mathcal{ H } } \mathcal{ L }_h$, then assuming that there exists a probability measure and an expectation $\widehat{ \mathbb{ E } }$ with respect to that measure, such that the standard stationarity condition $\mathbb{ E }\left[ \mathcal{ L } f( \mathbf{ X } ) \right ] = 0$ holds component-wise: 
\begin{equation*}
\widehat{ \mathbb{ E } } \left[ \mathcal{ L }_h f( \mathbf{ X } ) \right] = 0 \text{ for every } h \in \mathcal{ H } \text{ and  } f \in C^2( \Delta_{ \mathcal{ H } } ).
\end{equation*}
Substituting the probability of an ordered sample $q( \mathbf{ n } | \mathbf{ x } ) = \prod_{ h \in \mathcal{ H } } x_h^{ n_ h }$ yields a recursion whose solution is defined as the approximate CSD.
The same argument can be applied to the $\Lambda$-Fleming-Viot process to define approximate CSDs for the $\Lambda$-coalescent.
\vskip 11pt
\begin{thm}
Let $\hat{ \pi }( \mathbf{ m } | \mathbf{ n } )$ denote the approximate $\Lambda$-coalescent CSD as defined above.
It solves the following recursion
\begin{align}
m& \left[ \frac{ \Lambda( \{ 0 \} ) ( n + m - 1 ) }{ 2 } + \theta + \frac{ 1 }{ n + m } \sum_{ k = 2 }^{ n + m } \binom{ n + m }{ k } \lambda_{ n + m, k } \right] \hat{ \pi }( \mathbf{ m } | \mathbf{ n } ) \nonumber \\
=& \sum_{ h \in \mathcal{ H } } m_h \Bigg[ \frac{ \Lambda( \{ 0 \} ) ( n_h + m_h - 1 ) }{ 2 } \hat{ \pi }( \mathbf{ m } - \mathbf{ e }_h | \mathbf{ n } ) +  \sum_{ l \in L } \theta_l  \sum_{ a \in E_l } P_{ a h[ l ] }^{ ( l ) } \hat{ \pi }( \mathbf{ m } - \mathbf{ e }_h + \mathbf{ e }_{ S_l^a( h ) } | \mathbf{ n } ) \nonumber \\
&+ \frac{ 1 }{ n_h + m_h } \Bigg\{ \sum_{ k = 2 }^{ m_h + 1 } \binom{ n_h + m_h }{ k } \lambda_{ n + m, k } \hat{ \pi }( \mathbf{ m } - ( k - 1 ) \mathbf{ e }_h | \mathbf{ n } ) \nonumber \\ 
&+ \sum_{ k = m_h + 2 }^{ n_h + m_h } \binom{ n_h + m_h }{ k } \lambda_{ n + m, k } \frac{ \hat{ \pi }( \mathbf{ m } - m_h \mathbf{ e }_h | \mathbf{ n } - ( k - m_h - 1 ) \mathbf{ e }_h ) }{ \hat{ \pi }( ( k - m_h - 1 ) \mathbf{ e }_h | \mathbf{ n } - ( k - m_h - 1 ) \mathbf{ e }_h ) } \Bigg\} \Bigg]. \label{a_csd}
\end{align}
\end{thm}
\begin{proof}
The generator of the $\Lambda$-Fleming-Viot jump-diffusion can be written as
\begin{align}
\mathcal{ L }f( \mathbf{ x } ) =& \sum_{ h \in \mathcal{ H } } \frac{ \Lambda( \{ 0 \} ) x_h }{ 2 } \sum_{ h' \in \mathcal{ H } } ( \delta_{ h h' } - x_{ h' } ) \frac{ \partial^2 }{ \partial x_h \partial x_{ h' } } f( \mathbf{ x } ) \nonumber \\
&+ \sum_{ h \in \mathcal{ H } } \sum_{ l \in L } \theta_l \sum_{ a \in E_l } x_{ S_l^a( h ) } \left( P^{ ( l ) }_{ a h[ l ] } - \delta_{ a h[ l ] } \right) \frac{ \partial }{ \partial x_h } f( \mathbf{ x } ) \nonumber \\
&+ \sum_{ h \in \mathcal{ H } } x_h \int_{ ( 0, 1 ] } \left\{ f( ( 1 - r ) \mathbf{ x } + r \mathbf{ e }_h ) - f( \mathbf{ x } ) \right\} r^{ -2 } \Lambda( dr ) =: \sum_{ h \in \mathcal{ H } } \mathcal{ L }_h f( \mathbf{ x } ). \label{sum_of_generators}
\end{align}
Substituting $q( \mathbf{ n } | \mathbf{ x } )$ yields the following three terms on the R.H.S.
\begin{align}
&\sum_{ h \in \mathcal{ H } } \frac{ \Lambda( \{ 0 \} ) n_h }{ 2 } \left[ ( n_h - 1 ) q( \mathbf{ n } + \mathbf{ e }_h | \mathbf{ x } ) - \sum_{ h ' \in \mathcal{ H } }( n_{ h' } - \delta_{ h h' } ) q( \mathbf{ n } | \mathbf{ x } ) \right] \label{kingman}\\
&+ \sum_{ h \in \mathcal{ H } } n_h \sum_{ l \in L } \theta_l \left[ \sum_{ a \in E_l } P_{ a h[ l ] }^{ ( l ) } q( \mathbf{ n } - \mathbf{ e }_h + \mathbf{ e }_{ S_l^a( h ) } | \mathbf{ x } ) - q( \mathbf{ n } | \mathbf{ x } ) \right] \label{mutation}\\
&+ \int_{ ( 0, 1 ] } \Bigg\{ \sum_{ h \in \mathcal{ H } } \sum_{ k = 0 }^{ n_h } \binom{ n_h }{ k } r^k ( 1- r )^{ n -  k } q( \mathbf{ n } - ( k - 1 ) \mathbf{ e }_h | \mathbf{ x } ) \nonumber \\
&\;\;\;\;\;\;\;\;\;\;\;\;\;\;\;\;\; - \sum_{ k = 0 }^n \binom{ n }{ k } r^k ( 1 - r )^{ n - k } q( \mathbf{ n } | \mathbf{ x } )\Bigg\} r^{ -2 } \Lambda( dr ). \nonumber
\end{align}
The $k = 0$ terms inside the integral cancel because $\sum_{ h \in \mathcal{ H } } x_h = 1$ and the $k = 1$ terms cancel because $\sum_{ h \in \mathcal{ H } } n_h = n$, which means the third summand can be written
\begin{equation}\label{lambda}
\sum_{ h \in \mathcal{ H } } \left\{ \sum_{ k = 2 }^{ n_h } \binom{ n_h }{ k } \lambda_{ n, k } q( \mathbf{ n } - ( k - 1 ) \mathbf{ e }_h | \mathbf{ x } ) - \frac{ n_h }{ n } \sum_{ k = 2 }^n \binom{ n }{ k } \lambda_{ n, k } q( \mathbf{ n } | \mathbf{ x } ) \right\}.
\end{equation}
Substituting \eqref{kingman}, \eqref{mutation} and \eqref{lambda} into \eqref{sum_of_generators} and rearranging gives
\begin{align}
&\sum_{ h \in \mathcal{ H } }\left[ \frac{ \Lambda( \{ 0 \} ) ( n - 1 ) }{ 2 } + \theta + \frac{ 1 }{ n }\sum_{ k = 2 }^n \binom{ n }{ k } \lambda_{ n, k } \right] q( \mathbf{ n } | \mathbf{ x } ) \nonumber\\
&= \sum_{ h \in \mathcal{ H } } \Bigg\{ \frac{ \Lambda( \{ 0 \} ) ( n_h -1 ) }{ 2 } q( \mathbf{ n } - \mathbf{ e }_h | \mathbf{ x } ) + \sum_{ l \in L } \theta_l \sum_{ a \in E_l } P_{ a h[ l ] }^{ ( l ) } q( \mathbf{ n } - \mathbf{ e }_h + \mathbf{ e }_{ S_l^a( h ) } | \mathbf{ x } ) \nonumber \\
&\;\;\;\;\;\;\;\;\;\;\;\;\;\;\; + \frac{ 1 }{ n_h }\sum_{ k = 2 }^{ n_h } \binom{ n_h }{ k } \lambda_{ n, k } q( \mathbf{ n } - ( k - 1 ) \mathbf{ e }_h | \mathbf{ x } ) \Bigg\}. \label{full_recursion}
\end{align}
The component-wise vanishing property implies
\begin{equation*}
\widehat{ \mathbb{ E } }\left[ \sum_{ h \in \mathcal{ H } } m_h \mathcal{ L }_h q( \mathbf{ n } | \mathbf{ X } ) \right] = \sum_{ h \in \mathcal{ H } } m_h \widehat{ \mathbb{ E } }\left[ \mathcal{ L }_h q( \mathbf{ n } | \mathbf{ X } ) \right] = 0,
\end{equation*}
so that \eqref{full_recursion} becomes
\begin{align*}
&m \left[ \frac{ \Lambda( \{ 0 \} ) ( n - 1 ) }{ 2 } + \theta + \frac{ 1 }{ n }\sum_{ k = 2 }^n \binom{ n }{ k } \lambda_{ n, k } \right] \widehat{ \mathbb{ E } } \left[ q( \mathbf{ n } | \mathbf{ X } ) \right] \\
&= \sum_{ h \in \mathcal{ H } } m_h \Bigg\{ \frac{ \Lambda( \{ 0 \} ) ( n_h - 1 ) }{ 2 } \widehat{ \mathbb{ E } } \left[ q( \mathbf{ n } - \mathbf{ e }_h | \mathbf{ X } ) \right] + \sum_{ l \in L } \theta_l  \sum_{ a \in E_l } P_{ a h[ l ] }^{ ( l ) } \widehat{ \mathbb{ E } }\left[ q( \mathbf{ n } - \mathbf{ e }_h + \mathbf{ e }_{ S_l^a( h ) } | \mathbf{ X } ) \right]  \\
&\;\;\;\;\;\;\;\;\;\;\;\;\;\;\;\;\;\;\;\; + \frac{ 1 }{ n_h } \sum_{ k = 2 }^{ n_h } \binom{ n_h }{ k } \lambda_{ n, k } \widehat{ \mathbb{ E } }\left[ q( \mathbf{ n } - ( k - 1 ) \mathbf{ e }_h | \mathbf{ X } ) \right] \Bigg\}.
\end{align*}
Note that $\pi( \mathbf{ m } | \mathbf{ n } ) = \mathbb{ E }\left[ q( \mathbf{ n } + \mathbf{ m } | \mathbf{ X } ) \right] / \mathbb{ E }\left[ q( \mathbf{ n } | \mathbf{ X } ) \right]$ so that substituting $\mathbf{ n } \mapsto \mathbf{ n } + \mathbf{ m }$, assuming that $\mathbb{ E } = \widehat{ \mathbb{ E } }$ and dividing by $\mathbb{ E }\left[ q( \mathbf{ n } | \mathbf{ X } ) \right]$ gives the desired recursion.
\end{proof}
\begin{cor}\label{univariate_cor}
The univariate approximate CSDs $\hat{ \pi }( \mathbf{ e }_h | \mathbf{ n } )$ satisfy
\begin{align}
&\left[ \frac{ \Lambda( \{ 0 \} ) n }{ 2 } + \theta + \frac{ 1 }{ n + 1 } \sum_{ k = 2 }^{ n + 1 } \binom{ n + 1 }{ k } \lambda_{ n + 1, k } \right] \hat{ \pi }( \mathbf{ e }_h | \mathbf{ n } ) = \frac{ n_h }{ 2 }\left( \Lambda( \{ 0 \} ) + \lambda_{ n + 1, 2 } \right) \nonumber \\
&+ \sum_{ l \in L } \theta_l  \sum_{ a \in E_l } P_{ a h[ l ] }^{ ( l ) } \hat{ \pi }( \mathbf{ e }_{ S_l^a( h ) } | \mathbf{ n } ) + \frac{ 1 }{ n_h + 1 } \sum_{ k = 3 }^{ n_h + 1 } \binom{ n_h + 1 }{ k } \frac{ \lambda_{ n + 1, k } }{ \hat{ \pi }( ( k - 2 ) \mathbf{ e }_h | \mathbf{ n } - ( k - 2 ) \mathbf{ e }_h ) }. \label{univariate_recursion}
\end{align}
\end{cor}
\begin{proof}
The result follows by substituting $\mathbf{ m } = \mathbf{ e }_h$ into \eqref{a_csd}.
\end{proof}
As per Remark \ref{univariate_remark} it is sufficient to work with the simpler recursion \eqref{univariate_recursion} as opposed to the full recursion \eqref{a_csd}.
However, because of the denominator in the final term of \eqref{univariate_recursion} the resulting system of equations still contains as many unknowns as the recursion for the full likelihood.
Hence further approximations are needed to obtain a family of proposal distributions which is feasible to evaluate and sample.
\vskip 11pt
\begin{defn}
Setting $\Lambda = \delta_0$ in \eqref{univariate_recursion} results in the approximate CSDs derived in \cite{Stephens00} for Kingman's coalescent.
This approximation ignores the dynamics of the $\Lambda$-coalescent but results in a valid IS proposal distribution that still simulates $\Lambda$-coalescent trees.
We denote this proposal distribution by $\mathbb{ Q }^{ \text{SD} }$.
\end{defn}

In \cite{Paul10} Paul and Song introduced the \emph{trunk ancestry}, which can be used to obtain an approximation which makes better use of the $\Lambda$-coalescent structure.
We briefly recall the definition of the trunk ancestry here before using it to define a second approximate CSD family.
\vskip 11pt
\begin{defn}
The trunk ancestry $\mathcal{ A }^*( \mathbf{ n } )$ of a sample $\mathbf{ n }$ is a deterministic, degenerate process started from $\mathbf{ n }$ and evolving backwards in time but undergoing no dynamics.
\end{defn}
In the trunk ancestry, the lineages that form $\mathbf{ n }$ do not mutate or coalesce, and hence do not reach a MRCA.
Instead they form an ancestral forest or ``trunk" that extends infinitely into the past.

The first two terms on the R.H.S.~of \eqref{univariate_recursion}, corresponding to pairwise mergers and mutations, can be interpreted genealogically as the rates with which the $(n+1)^{ \text{th} }$ lineage mutates and is absorbed into $\mathcal{ A }^*( \mathbf{ n } )$ by a pairwise merger.
The third term corresponds to a multiple merger between the $( n + 1 )^{ \text{th} }$ lineage and two or more lineages in $\mathbf{ n }$ of the same type.
Because this last term involves coalescence between lineages in $\mathbf{ n }$ it does not have an interpretation in terms of $\mathcal{ A }^*( \mathbf{ n } )$.
However it can be forced into this framework by noting that the only relevant information is the time of absorption of the $( n+1)^{ \text{th} }$ lineage and the type of the lineage(s) in $\mathbf{ n }$ with which it merges.
Motivated by the trunk ancestral interpretation we expect the following recursion to be a good, tractable approximation to \eqref{univariate_recursion}.
\vskip 11pt
\begin{defn}
Let $\hat{ \pi }^K( \mathbf{ e }_h | \mathbf{ n } )$ be the distribution of the type of a lineage which, when traced backwards in time, mutates with rates $\theta_l$ according to the transition matrix $P^{ ( l ) }$ at each locus $l \in L$ and is absorbed into $\mathcal{ A }^*( \mathbf{ n } )$ with rate 
\begin{equation*}
\frac{ \Lambda( \{ 0 \} ) n }{ 2 } + \frac{ 1 }{ n + 1 }\sum_{ k = 2 }^{ n + 1 } \binom{ n + 1 }{ k } \lambda_{ n + 1, k },
\end{equation*}
choosing its parent uniformly upon absorption.
The corresponding IS proposal distribution is denoted by $\mathbb{ Q }^{ \text{K} }$.
\end{defn}
\vskip 11pt
\begin{prop}\label{lambda_acsd}
$\hat{ \pi }^K( \mathbf{ e }_h | \mathbf{ n } )$ satisfies the equations
\begin{align*}
\Bigg[ \frac{ \Lambda( \{ 0 \} ) n }{ 2 } + \theta + &\frac{ 1 }{ n + 1 }\sum_{ k = 2 }^{ n + 1 } \binom{ n + 1 }{ k } \lambda_{ n + 1, k } \Bigg] \hat{ \pi }^K( \mathbf{ e }_h | \mathbf{ n } ) = \frac{ \Lambda( \{ 0 \} ) n_h }{ 2 }\\
&+ \sum_{ l \in L } \theta_l \sum_{ a \in E_l } P_{ a h[ l ] }^{ ( l ) } \hat{ \pi }^K( \mathbf{ e }_{ S_l^a( h ) } | \mathbf{ n } )  + \frac{ n_h }{ n ( n + 1 ) } \sum_{ k = 2 }^{ n + 1 } \binom{ n + 1 }{ k } \lambda_{ n + 1, k } 
\end{align*}
and is the stationary distribution of the Markov chain on $\mathcal{ H }$ with transition matrix
\begin{equation}\label{markov_chain}
\frac{ P + \left[ \Lambda( \{ 0 \} ) / 2 + \frac{ 1 }{ n( n + 1 ) } \sum_{ k = 2 }^{ n + 1 } \binom{ n + 1 }{ k } \lambda_{ n + 1, k } \right] N }{ \sum_{ l \in L } \theta_l + \frac{ \Lambda( \{ 0 \} ) n }{ 2 } + \frac{ 1 }{ n + 1 } \sum_{ k = 2 }^{ n + 1 } \binom{ n + 1 }{ k } \lambda_{ n + 1, k } },
\end{equation}
where $N$ is the $| \mathcal{ H } | \times | \mathcal{ H } |$ matrix with each row equal to $( n_1, \ldots, n_{ | \mathcal{ H } | } )$ and $P$ is the transition probability matrix on $\mathcal{ H }$ formed as a mixture of the matrices $\{ P^{ ( l ) } \}_{ l \in L }$ with weights $\{ \theta_l \}_{ l \in L }$.
\end{prop}
\begin{proof}
The simultaneous equations follow by tracing the $( n + 1 )^{ \text{th} }$ lineage backwards in time and decomposing based on the first event, and the transition matrix follows immediately from the simultaneous equations.
\end{proof}
Note that $\mathbb{ Q }^{ \text{K} }$ has a very similar form to $\mathbb{ Q }^{ \text{SD} }$, and as a consequence of the linearity in $N$ in \eqref{markov_chain} the efficient Gaussian quadrature approximation of Appendix A in \cite{Stephens00} can be applied to both with minor modifications for $\mathbb{ Q }^{\text{K}}$.

\section{Simulation study}\label{simulation}

In this section we present an empirical comparison between the IS algorithms defined by $\mathbb{ Q }^{ \text{SD} }$ and $\mathbb{ Q }^{ \text{K} }$, and the generalised Griffiths-Tavar\'{e} proposal distribution from \cite{Birkner08} which will be denoted by $\mathbb{ Q }^{ \text{GT} }$.
We will also introduce two PAC algorithms making use of, respectively, $\hat{ \pi }^{ \text{K} }$ and a refinement to be specified below, and investigate their accuracy.
Simulated samples have been generated using the efficient sampling algorithm provided in Section 1.4.4 of \cite{Birkner09b}.
Approximate CSDs have been evaluated using a Gauss quadrature of order four (see Appendix A of \cite{Stephens00} for details).

Simulated chromosomes consist of 15 loci with two possible alleles denoted $\{ 0, 1 \}$ and mutation matrix $P^{ ( l ) } = \left( \begin{array}{cc} 0 & 1 \\ 1 & 0 \end{array} \right)$ at each locus.
The coalescent is a $\operatorname{Beta}( 2 - \alpha, \alpha )$-coalescent.
All simulations have been run on  a single core on a Toshiba laptop, and make use of a stopping time resampling scheme \cite{Jenkins12} with resampling checks made at hitting times of all sample sizes reaching $B = \{ n - 5, n - 10, \ldots, 5 \}$.
This generic resampling regime has been chosen for simplicity and without regard for any particular proposal distribution.

\subsection{Experiment 1}

The total mutation rate is $\theta = 0.1$ spread evenly among all 15 loci.
The coalescent is specified as $\alpha = 1.5$.
The data consists of 100 sampled chromosomes, 95 of which share a single type, four lineages a second type one mutation away from the main block, and a single lineage is of a third type one different mutation removed from the main block.

We consider inferring both $\theta$ and $\alpha$ individually, assuming all other parameters are known and that $\theta_l = \theta / L$ for every $l \in L$.
Eight independent simulations of 30 000 particles each were run on an evenly spaced grid of mutation rates spanning the interval $[ 0.025, 0.2 ]$.
The same simulations were then repeated on an evenly spaced grid spanning $\alpha \in [ 1.1125, 1.9 ]$.
The resulting likelihood surfaces are shown in Figure \ref{results_100_15}.
\begin{figure}[!ht]
\centering
\includegraphics[width = \linewidth]{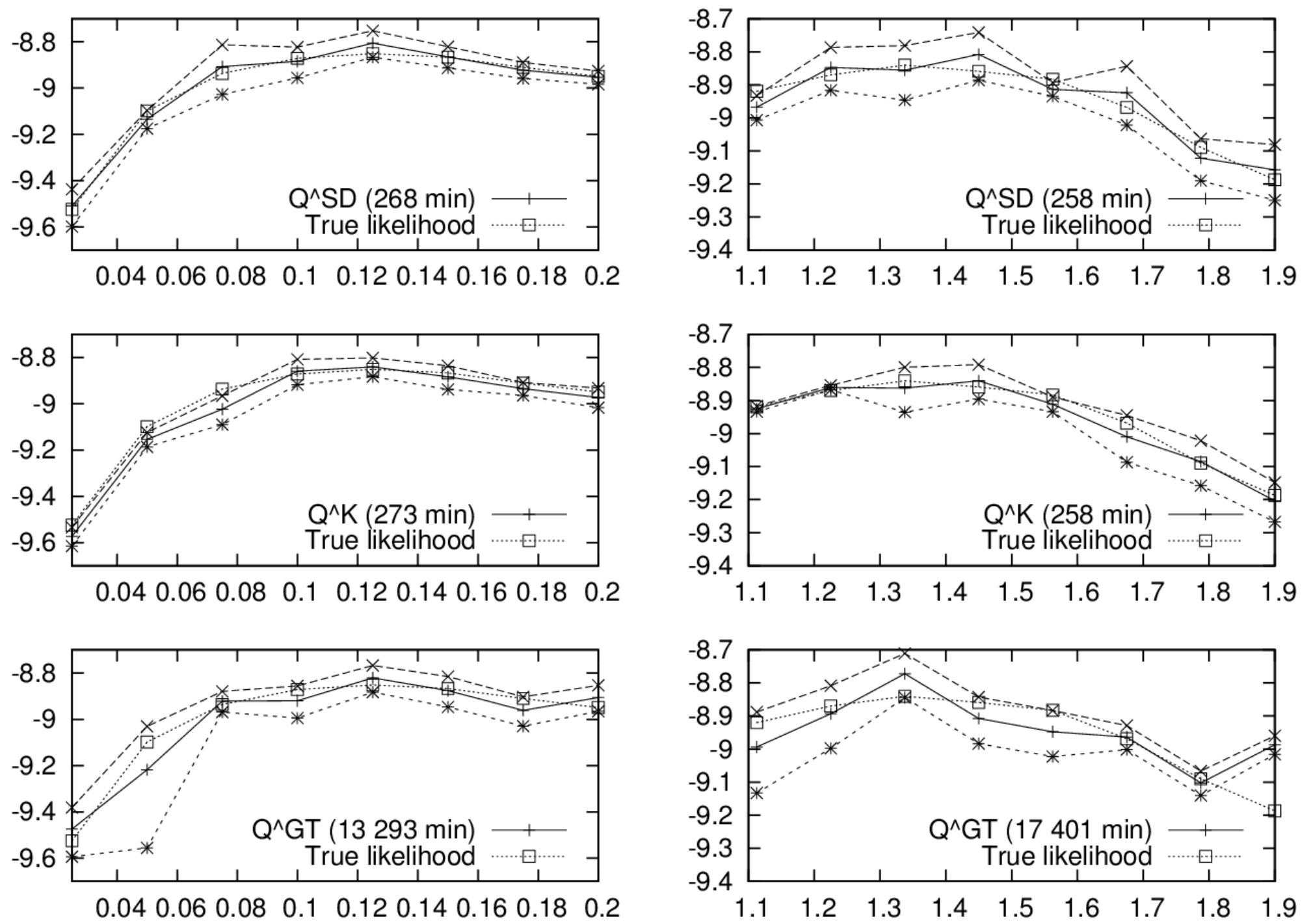}
\caption{Simulated log-likelihood surfaces from 30 000 particles with $\pm 2$SE confidence envelopes. The left column is for $\theta$ and the right for $\alpha$. The true surfaces are based on a 1 000 000 particle simulation using the $\mathbb{ Q }^{ \text{K} }$ proposal distribution.}
\label{results_100_15}
\end{figure}

The most striking observation is that both approximate CSDs yield algorithms which are two orders of magnitude faster than the Griffiths-Tavar\'{e} scheme.
Moreover, it is clear that the $\alpha$-surface obtained from $\mathbb{ Q }^{ \text{GT} }$ has not fully converged.
The wide confidence envelope at the left hand edge and the lack of monotonicity at the right hand edge of the $\mathbb{ Q }^{ \text{GT} }$ $\theta$-surface are indicative of poorer performance when inferring $\theta$ as well.

The runtimes of $\mathbb{ Q }^{ \text{SD} }$ and $\mathbb{ Q }^{ \text{K} }$ are very similar in both cases, and all four surfaces from these proposals are good approximations of the truth.
In the $\theta$-case the accuracy of the two is very similar, but in the $\alpha$-case $\mathbb{ Q }^{ \text{K} }$ yields noticeably tighter confidence bounds and a smoother surface.
This is particularly true of low values of $\alpha$, which correspond to Beta-coalescents that are very different from $\Lambda = \delta_0$.

Joint inference of $\alpha$ and $\theta$ is also of interest.
Figure \ref{results_surface} shows a joint likelihood heat map for the two parameters constructed from a grid of simulations of 30 000 particles from the $\mathbb{ Q }^{ \text{K} }$ proposal.
The surface is flat due to the limited amount of information in 100 samples, but the maximum likelihood estimator is close to the true $( 1.5, 0.1 )$ and the surface shows a high degree of monotonicity.
\begin{figure}[!ht]
\centering
\includegraphics[width = \linewidth]{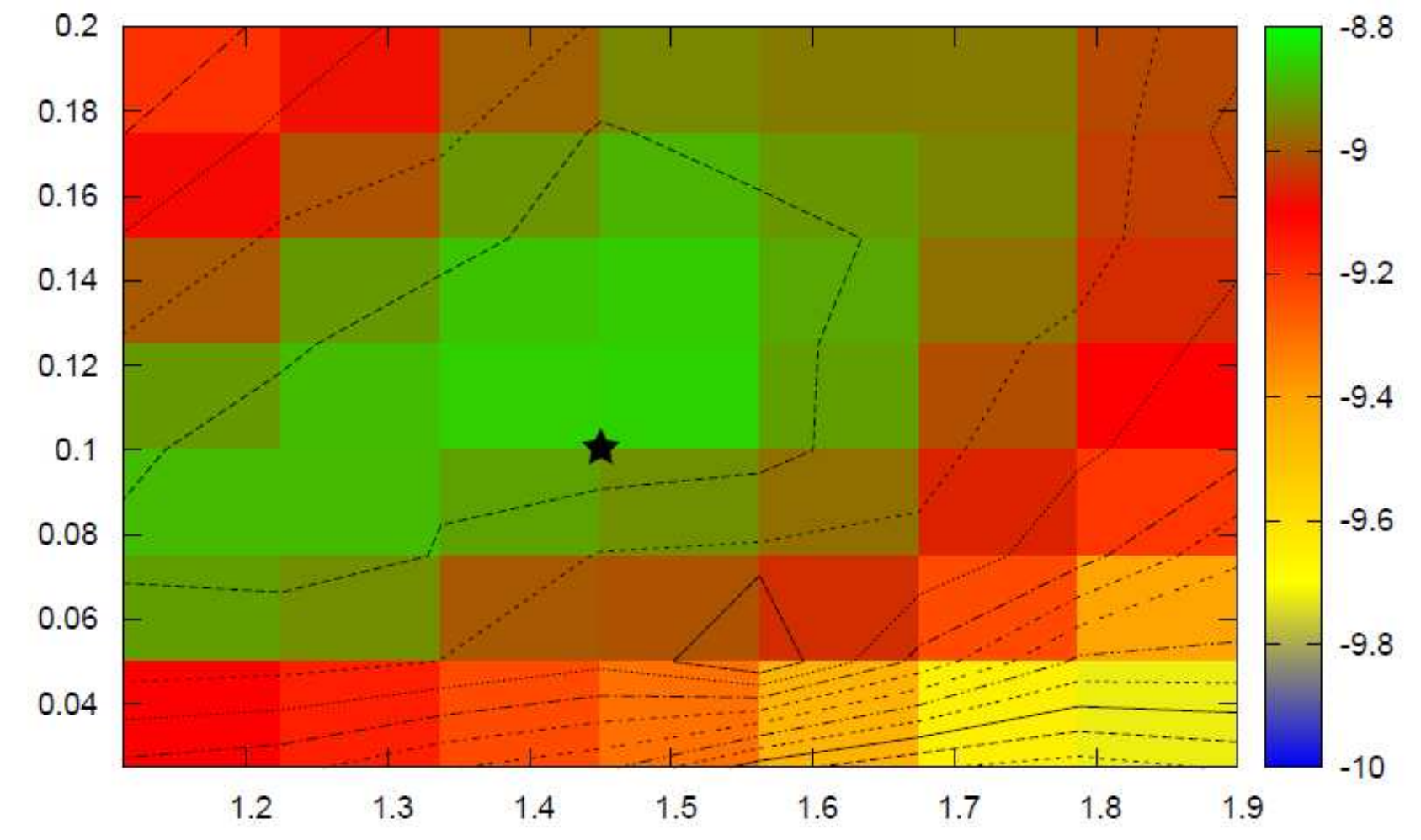}
\caption{Simulated likelihood surface from 30 000 particles using $\mathbb{ Q }^{ \text{K} }$. The surface is interpolated from an $8 \times 8$ grid of independent simulations.
The star denotes the MLE, which must lie on one of the grid points.}
\label{results_surface}
\end{figure}

\subsection{Experiment 2}

We expect the performance of the $\mathbb{ Q }^{ \text{SD} }$ proposal to deteriorate the further the true model is from Kingman's coalescent, and the more demanding the data set.
To that end the one dimensional inference problems for $\theta$ and $\alpha$ were repeated for a sample of 150 lineages with true parameters $\theta = 0.15$ and $\alpha = 1.2$.
The data set consists of 144 lineages of a given type with four other types present, each a single mutation removed from the main group.
The sizes of these groups are 3, 1, 1, 1.
The results are shown in Figure \ref{results_150_15}.
\begin{figure}[!ht]
\centering
\includegraphics[width = \linewidth]{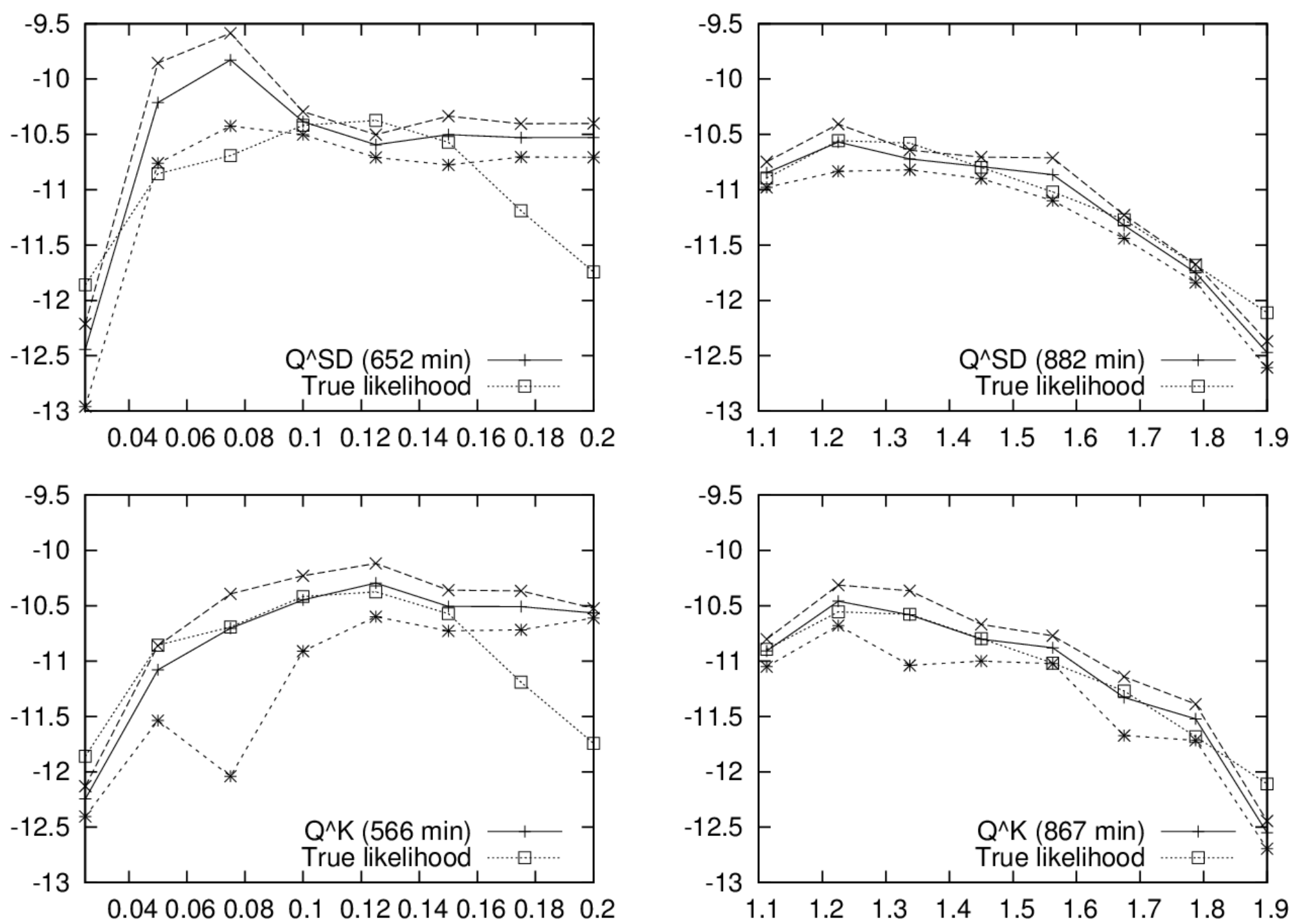}
\caption{Simulated log-likelihood surfaces from 30 000 particles with $\pm 2$SE confidence envelopes. The left column is for $\theta$ and the right for $\alpha$. The downward spike in the lower confidence boundary of the bottom left surface is an artifact caused by a negative value of the estimate, and the real part of the logarithm has been plotted. The values of the standard errors have no such spike.}
\label{results_150_15}
\end{figure}

$\mathbb{ Q }^{ \text{K} }$ is noticeably faster when inferring $\theta$, and slightly faster when inferring $\alpha$.
It also produces substantially more accurate estimates than $\mathbb{ Q }^{ \text{SD} }$ for small values of $\theta$.
30 000 particle runs have not yielded an accurate estimate for large values of $\theta$ from either algorithm.
The $\alpha$-surface from $\mathbb{ Q }^{ \text{SD} }$ looks superficially better, but both surfaces are very similar and good matches to the true likelihood.

This deterioration of the performance of $\mathbb{ Q }^{ \text{SD} }$ is to be expected because the true $\operatorname{Beta}( 0.8, 1.2 )$-coalescent is a significant departure from the $\Lambda = \delta_0$ assumption used to derive the corresponding approximate CSDs.
Such coalescents are of particular interest because significantly more efficient implementations exist for Kingman's coalescent, and these should be preferred whenever the Kingman hypothesis of $\Lambda = \delta_0$ cannot be rejected.
It is likely that the overestimated likelihood near $\theta = 0.06$ coincides with the MLE for this data set, had it been generated by Kingman's coalescent.
Hence $\mathbb{ Q }^{ \text{K} }$ is the recommended proposal distribution in practice.

Based upon the reported run times in Figures \ref{results_100_15} and \ref{results_150_15}, we expect our IS algorithm to be feasible for samples containing hundreds of lineages formed of tens of loci, or an order of magnitude more if methods such as a driving value \cite{Griffiths94c} or bridge sampling \cite{Meng96} are employed to reduce the number of independent simulations.
There is also a strong dependence on model parameters: fast coalescence (or low $\alpha$ in our setting) corresponds to faster simulation runs, and both high mutation rate and large haplotype space will result in a slower algorithm.

\subsection{Experiment 3: product of approximate conditionals}\label{exp3}

The IS algorithms used in the previous numerical experiments provide accurate results with reasonable computational cost, but the inference problems and data sets are of toy size and complexity.
It is clear that these algorithms will be too slow for many problems of interest, such as genome-wide data or large sample sizes.
The PAC method is a principled way of overcoming this restriction at the cost of asymptotic correctness, but with very significant improvements in speed.
It is based on decomposing the likelihood of observed alleles $h_1, \ldots, h_n$ into a product of CSDs:
\begin{equation*}
\mathbb{ P }( h_1, h_2, \ldots, h_n ) = \pi( h_n | h_1, \ldots, h_{ n - 1 } ) \pi( h_{ n - 1 } | h_1, \ldots, h_{ n - 2 } ) \times \ldots \times \pi( h_2 | h_1 ) \pi( h_1 )
\end{equation*}
and then substituting in a tractable, approximate CSD to obtain computable estimates.
We consider two different classes of CSDs: $\hat{ \pi }^{ \text{K} }( \cdot | \mathbf{ n } )$, and a modified version in which a lineage is absorbed into $\mathcal{ A }^*( \mathbf{ n } )$ with rate
\begin{equation*}
\sum_{ h \in \mathcal{ H } } \left\{ \frac{ \Lambda( \{ 0 \} ) n_h }{ 2 } + \frac{ 1 }{ n_h + 1 }\sum_{ k = 2 }^{ n_h + 1 } \binom{ n_h + 1 }{ k } \lambda_{ n + 1, k } \right\}
\end{equation*}
and inherits the type of the cluster $n_h$ into which it is absorbed.
This approximate CSD will be denoted by $\hat{ \pi }^{ \text{K2} }( \cdot | \mathbf{ n } )$.
Note that because $\hat{ \pi }^{ \text{K2} }( \cdot | \mathbf{ n } )$ depends nonlinearly on the exact frequencies $( n_h )_{ h \in \mathcal{ H } }$, the precomputations which are possible for all other CSDs considered in this paper are not possible (see Proposition 1 of \cite{Stephens00} for details).
For an IS algorithm this loss of efficiency in evaluating the CSDs would be devastating, but PAC algorithms are fast enough to remain feasible.
The order of the Gauss quadrature used to approximate the CSDs has also been increased to 10 for both families.

Neither approximate CSD family is exchangeable, so the estimates of the likelihood depend on the order in which the count data $\mathbf{ n }$ is conditioned upon.
Following the approach of Li and Stephens, and subsequent works making use of the PAC method, we partially address this issue by averaging our estimates across 1 000 uniformly random permutations of the data.
The number of permutations is substantially larger than what has been used for PAC models based on Kingman's coalescent, but has proven necessary in trial runs (results not shown) and comes at little additional cost.
The results of applying these PAC algorithms to  both the individual and joint inference questions posed in Experiments 1 and 2 are summarised in Figures \ref{results_pac} and \ref{results_surf_pac}.

\begin{figure}[!ht]
\centering
\includegraphics[width = \linewidth]{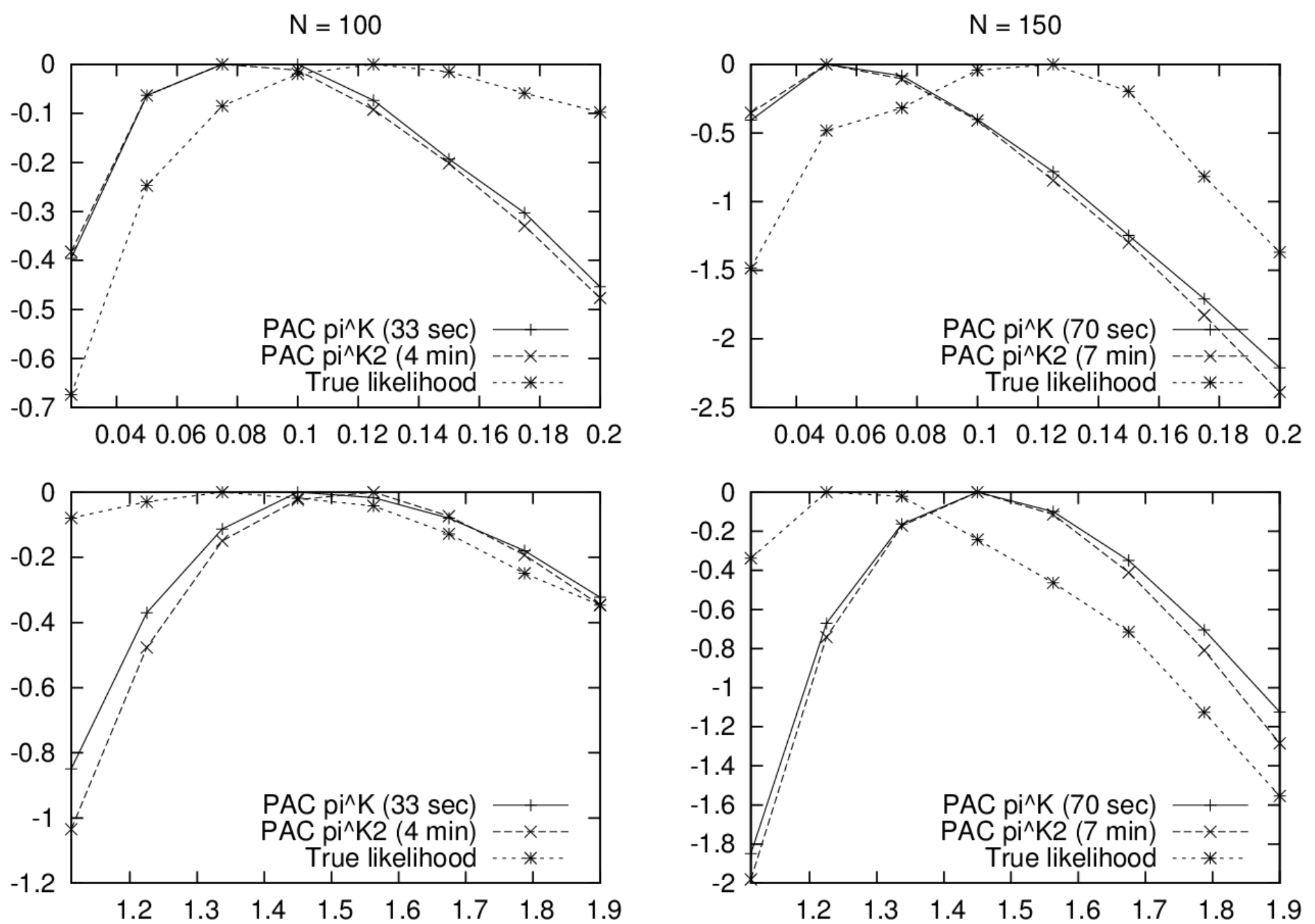}
\caption{The PAC log-likelihood surfaces normalised to 0 following Li and Stephens \cite{Li03}. The true likelihood surfaces in the left column are from Experiment 1, and the true surfaces in the right column are from Experiment 2.}
\label{results_pac}
\end{figure}

\begin{figure}[!ht]
\centering
\includegraphics[width = \linewidth]{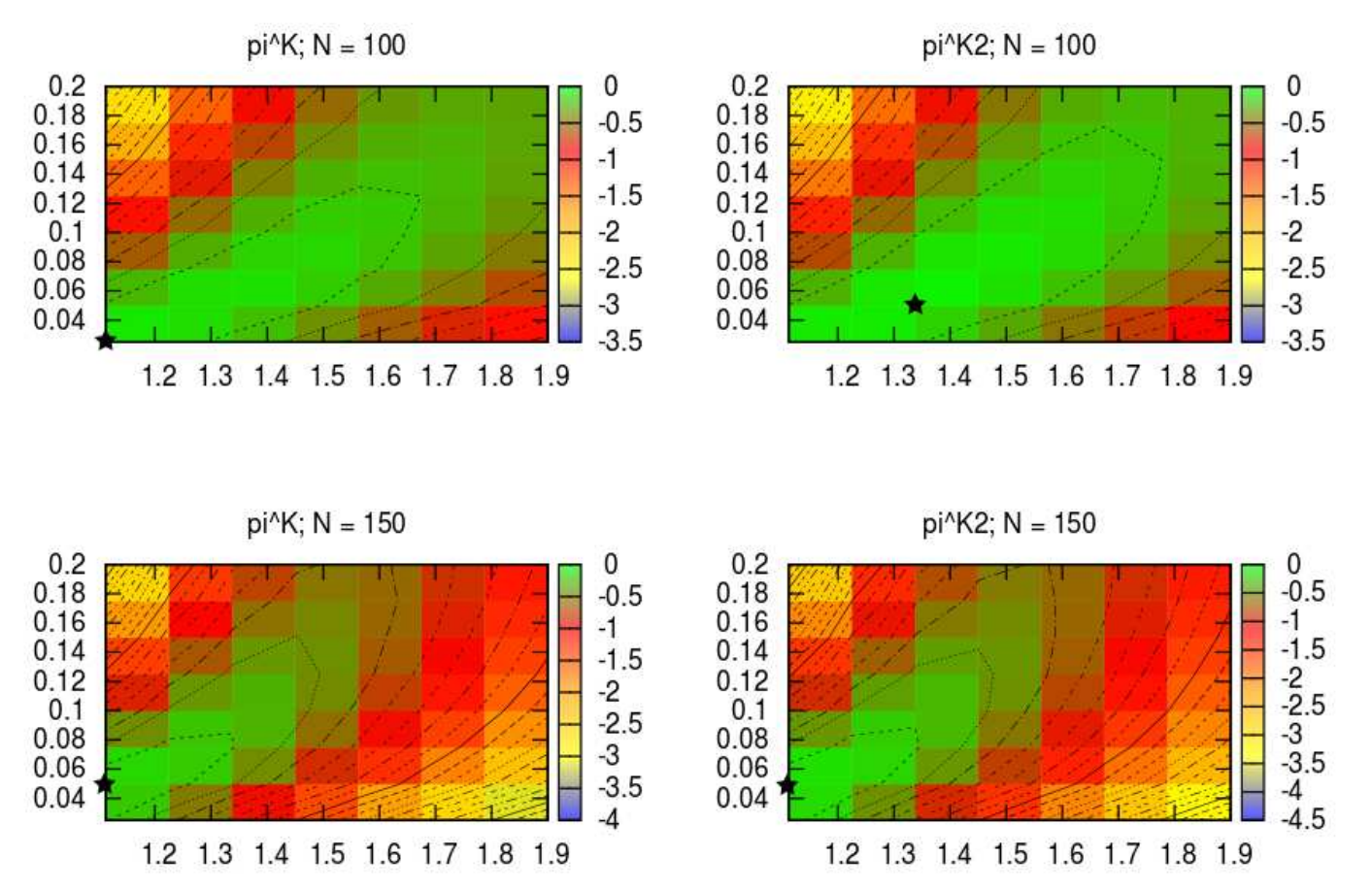}
\caption{The PAC joint log-likelihood surfaces normalised to 0. Locations of MLEs are indicated by stars. Figure \ref{results_surface} provides a suitable IS comparison to the top row.}
\label{results_surf_pac}
\end{figure}

The results of the PAC calculations seem mixed.
Both PAC algorithms are extremely fast, and would likely remain feasible even for reasonably large data sets.
The PAC likelihood estimates are consistently too low by many orgers of magnitude.
Nevertheless, the PAC MLEs in Figure \ref{results_pac} are strikingly close to the true maximisers, particularly for the smaller data set in the left column.
On the other hand, the joint MLEs in Figure \ref{results_surf_pac} are much further from the truth, although the surfaces still broadly capture the diagonal shape seen in Figure \ref{results_surface}.
It is also interesting to note that the two PAC methods perform very similarly in the one-dimensional problems in Figure \ref{results_pac}, but the 2D surface obtained from $\pi^{ \text{K2} }$ is a better fit than that from $\pi^{ \text{K} }$ for the smaller sample. 
For the larger sample the surfaces are nearly identical.

The run times in Figure \ref{results_pac} indicate that the PAC method will remain computationally feasible for substantially larger data sets than the IS algorithm, at least up to tens of thousands of lineages and/or thousands of loci.
Of course, the accuracy of the PAC method to such data sets cannot be concluded from the trials presented here, and careful verification will be necessary on a case-by-case basis.
In further contrast to IS, the runtime of the PAC algorithm is independent of the model parameters, and influenced only weakly by the size of the space of haplotypes.

A substantial amount of work will be required to develop a thorough understanding of the accuracy and pitfalls of these PAC algorithms, and whether or not the more advanced PAC algorithms developed for Kingman's coalescent can be adapted to the $\Lambda$-coalescent setting as well.
Our preliminary simulations motivate this undertaking, and confirm that the PAC method can provide useful, principled and fast results for $\Lambda$-coalescents.

\section{Importance sampling for \texorpdfstring{$\Xi$}{Xi}-coalescents}\label{xi_coalescents}

The important tools in deriving the optimal proposal distributions $\mathbb{ Q }^*( \cdot | \cdot )$ and the approximate CSDs $\hat{ \pi }^{ K }( \cdot | \cdot )$ were, respectively, the lookdown construction of $\cite{Donnelly99}$ and the trunk ancestry of $\cite{Paul10}$.
Both of these are also available for the $\Xi$-coalescent, and in this section we make use of them to extend the IS algorithm to this family of models.

A lookdown construction for the $\Xi$-coalescent and the $\Xi$-Fleming-Viot process was derived by Birkner et al. \cite{Birkner09a} and can be described as follows.
For ease of notation we assume $\Xi( \{ \mathbf{ 0 } \} ) = 0$.
If $\Xi$ does have an atom at zero, its treatment is identical to the $\Lambda$-case.

Let $\Pi_{ \Xi }$ be a Poisson point process on on $\mathbb{ R }_+ \times \Delta \times [ 0,1 ]^{ \mathbb{ N } }$ with rate 
\begin{equation*}
dt \otimes \left( \sum_{ i = 1 }^{ \infty } r_i^2 \right)^{ -1 } \Xi( d\mathbf{ r } ) \otimes dr^{ \otimes \mathbb{ N } }
\end{equation*}
and associate to each lineage a level $\{ 1, \ldots, n \}$.
Define the function
\begin{equation*}
g( \mathbf{ r }, u ) :=
\begin{cases}
\min\left\{ j \in \mathbb{ N } : \sum_{ i = 1 }^j r_i \geq u \right\} &\text{ if } u \leq \sum_{ i = 1 }^{ \infty } r_i \\
\infty &\text{ otherwise }
\end{cases}.
\end{equation*}
At each $( t_j, ( r_{ j 1 }, r_{ j 2 }, \ldots ), ( u_{ j 1 }, u_{ j 2 }, \ldots ) ) \in \Pi_{ \Xi }$ group the $n$ particles such that all particles $l \in \{ 1, \ldots, n \}$ with $g( \mathbf{ r }_j, u_{ j l } ) = k$ form a family for each $k \in \mathbb{  N }$.
Among each family every particle copies the type of the particle with the lowest level.
In addition each particle follows an independent mutation process as for the $\Lambda$-coalescent.

This lookdown construction will be instrumental in establishing the following recursion, which is an explicit version of \eqref{sequential_likelihood} for $\Xi$-coalescents and  a finite sites analogue of the sampling recursion presented in \cite{Moehle06} for the infinite alleles model.
\vskip 11pt
\begin{thm}
The likelihood of type frequencies $\mathbf{ n } \in \mathbb{ N }^{ | \mathcal{ H } | }$ sampled from the stationary $\Xi$-Fleming-Viot process solves
\begin{align}
\mathbb{ P }( \mathbf{ n } ) = &\frac{ 1 }{ g_n + n \theta } \Bigg\{ \sum_{ h : n_h > 0 } \sum_{ l \in L } \theta_l \sum_{ a \in E_l } \left( n_{ S_l^a( h ) } + 1 - \delta_{ a h[ l ] } \right) \mathbb{ P }( \mathbf{ n } - \mathbf{ e }_h + \mathbf{ e }_{ S_l^a( h ) } ) \label{xi_recursion}\\
&+ \sum_{ k_1 = 1 }^{ n_1 }  \ldots \sum_{ k_{ | \mathcal{ H } | } = 1 }^{ n_{ | \mathcal{ H } | } } \sum_{ \pi^1 \in P_{ n_1 }^{ k_1 } } \ldots \sum_{ \pi^{ | \mathcal{ H } | } \in P_{ n_{ | \mathcal{ H } | } }^{ k_{ | \mathcal{ H } | } } } \mathds{ 1 }\left\{ \sum_{ h \in \mathcal{ H } } k_h < n \right\}  \binom{ n }{ | \pi_1^1 |, | \pi_2^1 |, \ldots, | \pi_{ | \mathcal{ H } | }^{ | \mathcal{ H } | } | } \nonumber\\
&\times  \binom{ | \vee_{ h \in \mathcal{ H } }\pi^h | }{ | \pi^1 |, \ldots, | \pi^{ | \mathcal{ H } | } | }^{ -1 } \lambda_{ n; K( \vee_{ h \in \mathcal{ H } } \pi^h ); S( \vee_{ h \in \mathcal{ H } } \pi^h ) } \mathbb{ P }( \mathbf{ k } ) \Bigg\} \nonumber
\end{align}
with the convention that $\sum_{ k = 1 }^0 f( k ) = f( 0 )$.
Here $P_{ n_h }^{ k_h }$ denotes the set of equivalence relations on $n_h \in \mathbb{ N }$ elements with $k_h \leq n_h$ equivalence classes, $\pi^h = ( \pi_1^h \ldots \pi_{ k_h }^h )$ denotes such an equivalence relation so that $\sum_{ i = 1 }^{ k_h } | \pi^h_i | = n_h$ and $\vee_{ h \in \mathcal{ H } } \pi^h$ is the equivalence relation on $n$ elements obtained from appying each $\pi^h$ to the corresponding $n_h$ elements.
The vector $K( \pi )$ lists the sizes of all equivalence classes with more than one member, $S( \pi )$ is the number of classes with exactly one member and $g_n$ is the total coalescence rate of $n$ untyped lineages given by
\begin{equation*}
g_n = \sum_{ a = 1 }^{ n - 1 }\frac{ n! }{ a! } \sum_{ \substack{ b_1, \ldots, b_a \in \mathbb{ N } \\ b_1 + \ldots + b_a = n } } \frac{ \lambda_{ n; K( \mathbf{ b } ); S( \mathbf{ b } ) } }{ b_1 ! \times \ldots \times b_a ! }.
\end{equation*}
\end{thm}
\begin{proof}
The proof is the same as in Section 1.4.1 of \cite{Birkner09b}, adapted here from the $\Lambda$-coalescent to the $\Xi$-coalescent.
Let $p$ denote the distribution of the types of the first $n$ levels of the stationary lookdown construction.
Decomposing according to which event (whether mutation or a merger) occurred first when tracing backwards in time yields
\begin{align}
p( y_1, \ldots, y_n ) = \frac{ 1 }{ g_n + n \theta } \Bigg\{ &\sum_{ i = 1 }^n \sum_{ l \in L } \theta_l \sum_{ a \in E_l } P_{ a h[ y_i ] }^{ ( l ) } p( y_1, \ldots, y_{ i - 1 }, S_l^a( y_i ), y_{ i + 1 }, \ldots, y_n ) \nonumber \\
& + \sum_{ \pi \in P( \mathbf{ y } ) } \lambda_{ n; K( \pi ); S( \pi ) } p( \gamma_{ \pi }( y_1, \ldots, y_n ) ) \Bigg\} \label{xi_lookdown}
\end{align}
where $P( \mathbf{ y } )$ is the set of equivalence relations describing permissible mergers for the sample $\mathbf{ y } = ( y_1, \ldots, y_n )$ (that is, mergers where no equivalence class contains lineages of more than one type) and $\gamma_{ \pi }( y_1, \ldots, y_n )$ is the vector of types which results in $( y_1, \ldots, y_n )$ if the look-down-and-copy event denoted by the equivalence relation $\pi$ takes place.

By exchangeability we are only interested in the vector of type frequencies $\mathbf{ n } = ( n_1, \ldots, n_{ | \mathcal{ H } | } )$.
For such a vector define the canonical representative as 
\begin{equation*}
\kappa( \mathbf{ n } ) := ( \underbrace{ 1, \ldots, 1 }_{ n_1 } , \underbrace{ 2, \ldots, 2 }_{ n_2 }, \ldots, \underbrace{ | \mathcal{ H } |, \ldots, | \mathcal{ H } | }_{ n_{ | \mathcal{ H } | } } ) 
\end{equation*}
and the likelihood as
\begin{equation*}
p^0( \mathbf{ n } ) := \binom{ n }{ n_1, \ldots, n_{ | \mathcal{ H } | } } p( \kappa( \mathbf{ n } ) ).
\end{equation*}
Now we have the following identities
\begin{align*}
n_h \binom{ n }{ n_1, \ldots, n_{ | \mathcal{ H } | } } p( \kappa( \mathbf{ n } - \mathbf{ e }_h + \mathbf{ e }_{ S_l^a( h ) } ) ) =  &( n_{ S_l^a( h ) } + 1 - \delta_{ a h[ l ] } ) p^0( \mathbf{ n } - \mathbf{ e }_h + \mathbf{ e }_{ S_l^a( h ) } ) \\
\binom{ n }{ n_1, \ldots, n_{ | \mathcal{ H } | } } \prod_{ h \in \mathcal{ H } } \binom{ n_h }{ | \pi_1^h |, \ldots, | \pi_{ k_h }^h | } p( \kappa( \mathbf{ k } ) ) = &\binom{ n }{ | \pi_1^1 |, | \pi_2^1 |, \ldots, | \pi_{ | \mathcal{ H } | }^{ | \mathcal{ H } | } | } \binom{ k }{ k_1, \ldots, k_{ | \mathcal{ H } | } }^{ -1 } p^0( \mathbf{ k } )
\end{align*}
which, when substituted into \eqref{xi_lookdown}, yield the desired recursion.
\end{proof}
As in Section \ref{optimal_proposal} we can consider approximating the solution to \eqref{xi_recursion} by importance sampling, and the following theorem is a straightforward extension of Theorem \ref{lambda_optimal_proposal}.
\vskip 11pt
\begin{thm}\label{xi_optimal_proposal}
The optimal proposal distributions for recursion \eqref{xi_recursion}, denoted $\mathbb{ Q }^*_{ \Xi }$, are
\begin{equation*}
\mathbb{ Q }_{ \Xi }^*( H_i | H_{ i + 1 } ) \propto \begin{cases}
n_h \theta_l \frac{ \pi( \mathbf{ e }_{ S_l^a( h ) } | H_{ i + 1 } - \mathbf{ e }_h ) }{ \pi( \mathbf{ e }_h | H_{ i + 1 } - \mathbf{ e }_h ) } P^{ ( l ) }_{ a h[ l ] } \text{ if } H_i = H_{ i + 1 } - \mathbf{ e }_h + \mathbf{ e }_{ S_l^a( h ) } \\
\displaystyle \sum_{ \pi^1 \in P_{ n_1 }^{ k_1 } } \ldots \sum_{ \pi^{ | \mathcal{ H } | } \in P_{ n_{ | \mathcal{ H } | } }^{ k_{ | \mathcal{ H } | } } } \prod_{ h \in \mathcal{ H } } \binom{ n_h }{ | \pi_1^h |, \ldots, | \pi_{ k_h }^h | } \frac{ \lambda_{ n; K( \vee_{ h \in \mathcal{ H } } \pi^h ); S(  \vee_{ h \in \mathcal{ H } } \pi^h ) } }{ \pi( \mathbf{ n } - \mathbf{ k } | \mathbf{ k } ) } \\ \text{if } H_{ i + 1 } = \mathbf{ n } \text{ and } H_i = \mathbf{ k } \text{ for } k_h = \{ 1, \ldots, n_h \} \text{ and } \sum_{ h \in \mathcal{ H } } k_h < n
\end{cases}
\end{equation*}
where $n$ and $n_h$ denote type frequencies of $H_{ i + 1 }$.
\end{thm}
\begin{proof}
The argument is identical to the proof of Theorem \ref{lambda_optimal_proposal} taking into account the larger class of permitted simultaneous multiple mergers and hence different combinatorial coefficients.
\end{proof}

As before the CSDs used in the statement of Theorem \ref{xi_optimal_proposal} are not available, but any approximation to them will yield an unbiased algorithm and better approximations can be expected to correspond to more efficient algorithms.
The generator of the $\Xi$-Fleming-Viot process is not as immediately tractable as its Fleming-Viot and $\Lambda$-Fleming-Viot counterparts, so we abandon the generator-based approach of De Iorio and Griffiths and derive approximate CSDs from the trunk ancestry $\mathcal{ A }^*( \mathbf{ n } )$.
\vskip 11pt
\begin{defn}
Let $\hat{ \pi }^K_{ \Xi }( \mathbf{ e }_h | \mathbf{ n } )$ be the CSD obtained by letting the $( n + 1 )^{ \text{th} }$ lineage mutate with rates $\{ \theta_l \}_{ l \in L }$ via transition matrices $\{ P^{ ( l ) } \}_{ l \in L }$, be absorbed into $\mathcal{ A }^*( \mathbf{ n } )$ with rate
\begin{equation*}
\frac{ 1 }{ n + 1 } \sum_{ k = 1 }^n \sum_{ \pi \in P_{ n + 1 }^k } \binom{ n + 1 }{ | \pi_1 |, \ldots, | \pi_k | } \lambda_{ n + 1; K( \pi ); S( \pi ) },
\end{equation*}
and choose its parent uniformly upon absorption.
\vskip 11pt
\end{defn}
\begin{prop}
The approximate CSDs $\hat{ \pi }_{ \Xi }^K( \mathbf{ e }_h | \mathbf{ n } )$ solve the following recursion:
\begin{align*}
&\left[ \theta + \frac{ 1 }{ n + 1 }\sum_{ k = 1 }^n \sum_{ \pi \in P_{ n + 1 }^k } \binom{ n + 1 }{ | \pi_1 |, \ldots, | \pi_k | } \lambda_{ n + 1; K( \pi ); S( \pi ) } \right] \hat{ \pi }_{ \Xi }^K( \mathbf{ e }_h | \mathbf{ n } ) \\
&= \frac{ n_h }{ n ( n + 1 ) }\sum_{ k = 1 }^n \sum_{ \pi \in P_{ n + 1 }^k } \binom{ n + 1 }{ | \pi_1 |, \ldots, | \pi_k | } \lambda_{ n + 1; K( \pi ); S( \pi ) } + \sum_{ l \in L }\theta_l \sum_{ a \in E_l }P_{ a h[ l ] }^{ ( l ) } \hat{ \pi }_{ \Xi }^K( \mathbf{ e }_{ S_l^a( h ) } | \mathbf{ n } )
\end{align*}
and is the stationary distribution of the Markov Chain on $\mathcal{ H }$ with transition probability matrix
\begin{equation*}
\frac{ P + \left\{ \frac{ 1 }{ n ( n + 1 ) }\sum_{ k = 1 }^n \sum_{ \pi \in P_{ n + 1 }^k } \binom{ n + 1 }{ | \pi_1 |, \ldots, | \pi_k | }  \lambda_{ n + 1; K( \pi ); S( \pi ) } \right\} N }{ \theta + \frac{ 1 }{ n + 1 } \sum_{ k = 1 }^n \sum_{ \pi \in P_{ n + 1 }^k } \binom{ n + 1 }{ | \pi_1 |, \ldots, | \pi_k | }  \lambda_{ n + 1; K( \pi ); S( \pi ) } }.
\end{equation*}
where $P$ and $N$ are as in Proposition \ref{lambda_acsd}.
\end{prop}
\begin{proof}
The proof is identical to Proposition \ref{lambda_acsd} and follows by considering the first event backwards in time encountered by the lineage.
\end{proof}
Note that because simultaneous multiple mergers can take place, the decomposition in Remark \ref{univariate_remark} is no longer valid and multivariate approximate CSDs $\hat{ \pi }_{ \Xi }^K( \mathbf{ m } | \mathbf{ n } )$ must also be specified.
This is most naturally done by averaging over all permutations of the lineages in $\mathbf{ m }$, but this is computationally infeasible for all but very small samples $\mathbf{ m }$.
The PAC approach of averaging over a relatively small number of random permutations can be used to yield a more practical family, although algorithms will still be limited by the fact that evaluating the CSDs requires computing all equivalence classes on $n$ elements.
This burden can be alleviated considerably by assuming that the measure $\Xi$ places full mass on a finite dimensional simplex, which amounts to restricting the number of permitted simultaneous mergers to the same, finite number.
If this number is small compared to the size of the data set, far fewer terms will need to be computed at each stage of the algorithm but the model still allows for more general ancestral trees than any $\Lambda$-coalescent.
In particular, the case of up to four simultaneous mergers arising in coalescent models of diploid populations \cite{Moehle03}, \cite{Birkner13} seems computationally feasible.

\section{Discussion}\label{discussion}

In this paper we have developed novel IS algorithms for inference under the $\Lambda$- and $\Xi$-coalescent models, which retain the rigorous motivations of proposals that have been designed for Kingman's coalescent \cite{DeIorio04a}, \cite{DeIorio04b}, \cite{Paul10}. 
Furthermore, they outperform existing algorithms for $\Lambda$-coalescent inference, and like all IS methods are unbiased. 
It should be noted however that the greater modelling flexibility provided by $\Lambda$- and $\Xi$-coalescents comes with additional computational cost in comparison to the more restrictive Kingman’s coalescent. 
The inference problems considered in this paper have consisted of small samples of chromosomes comprised of a small number of loci, each with a simple mutation model. 
While some cost is certainly unavoidable, these computations can be sped up considerably by reducing the number of independent simulations through making use of driving values \cite{Griffiths94c}  or bridge sampling \cite{Meng96}.
It is also noteworthy that, as with IS algorithms in general, all of the algorithms used here can be parallelised very effectively.

The limits on data sets which can be feasibly analysed using IS are restrictive even under Kingmans coalescent, so alternate methods have been developed to tackle broader classes of problems.
The PAC method is a prime example, and our simulations suggest it is also a viable approach for $\Lambda$-coalescents.
Much work has been done on sophisticated approximations to CSDs for Kingman's coalescent with recombination and other features, and our results in Section \ref{exp3} indicate that investigating similar approaches under $\Lambda$- and $\Xi$-coalescents is a fruitful direction for future research.
Many of the generalisations of interest result in processes with generators that vary from those studied in this paper only by additive terms, so we expect that the machinery used here can be applied more generally with little added difficulty.

\section*{Acknowledgements}

Jere Koskela is a member of the MASDOC doctoral training centre at the University of Warwick, which is funded by Engineering and Physical Sciences Research Council grant EP/HO23364/1.

%
%
%
%

\bibliographystyle{alpha}
\bibliography{science}  

\newcommand{\etalchar}[1]{$^{#1}$}
\begin{thebibliography}{DIGLR05}

\bibitem[\'{A}04]{Arnason04}
E.~\'{A}rnason.
\newblock Mitochondrial cytochrome b {DNA} variation in the high--fecundity
  atlantic cod: trans--{A}tlantic clines and shallow gene genealogy.
\newblock {\em Genetics}, 166:1871--1885, 2004.

\bibitem[BB08]{Birkner08}
M.~Birkner and J.~Blath.
\newblock Computing likelihoods for coalescents with multiple collisions in the
  infinitely many sites model.
\newblock {\em J. Math. Biol.}, 57(3):435--463, 2008.

\bibitem[BB09]{Birkner09b}
M.~Birkner and J.~Blath.
\newblock Measure--valued diffusions, general coalescents and population
  genetic inference.
\newblock {\em in J. Blath, P. M\"{o}rters, M. Scheutzow (Eds.), Trends in
  Stochastic Analysis}, LMS 351:329--363, 2009.

\bibitem[BBB94]{Boom94}
J.D.G. Boom, E.G. Boulding, and A.T. Beckenback.
\newblock Mitochondrial {DNA} variation in introduced populations of {P}acific
  oyster, {C}rassostrea gigas, in {B}ritish {C}olumbia.
\newblock {\em Can. J. Fish. Aquat. Sci.}, 51:1608--1614, 1994.

\bibitem[BBE13]{Birkner13}
M.~Birkner, J.~Blath, and B.~Eldon.
\newblock An ancestral recombination graph for diploid populations with skewed
  offspring distribution.
\newblock {\em Genetics}, 193(1):255--290, 2013.

\bibitem[BBM{\etalchar{+}}09]{Birkner09a}
M.~Birkner, J.~Blath, M.~M\"{o}hle, M.~Steinr\"{u}cken, and J.~Tams.
\newblock A modified lookdown construction for the {Xi--Fleming--Viot} process
  with mutation and populations with recurrent bottlenecks.
\newblock {\em Alea}, 6:25--61, 2009.

\bibitem[BBS11]{Birkner11}
M.~Birkner, J.~Blath, and M.~Steinr\"{u}cken.
\newblock Importance sampling for {L}ambda--coalescents in the infinitely many
  sites model.
\newblock {\em Theor. Popln Biol.}, 79(4):155--173, 2011.

\bibitem[DIG04a]{DeIorio04a}
M.~De~Iorio and R.C. Griffiths.
\newblock Importance sampling on coalescent histories {I}.
\newblock {\em Adv. in Appl. Probab.}, 36(2):417--433, 2004.

\bibitem[DIG04b]{DeIorio04b}
M.~De~Iorio and R.C. Griffiths.
\newblock Importance sampling on coalescent histories {II}: {S}ubdivided
  population models.
\newblock {\em Adv. in Appl. Probab.}, 36(2):434--454, 2004.

\bibitem[DIGLR05]{DeIorio05}
M.~De~Iorio, R.C. Griffiths, L.~Leblois, and F.~Rousset.
\newblock Stepwise mutation likelihood computation by sequential importance
  sampling in subdivided population models.
\newblock {\em Theor. Popln Biol.}, 68:41--53, 2005.

\bibitem[DK99]{Donnelly99}
P.~Donnelly and T.~Kurtz.
\newblock Particle representations for measure--valued population models.
\newblock {\em Ann. Probab.}, 27(1):166--205, 1999.

\bibitem[EW06]{Eldon06}
B.~Eldon and J.~Wakeley.
\newblock Coalescent processes when the distribution of offspring number among
  individuals is highly skewed.
\newblock {\em Genetics}, 172:2621--2633, 2006.

\bibitem[FD01]{Fearnhead01}
P.~Fearnhead and P.~Donnelly.
\newblock Estimating recombination rates from population genetic data.
\newblock {\em Genetics}, 159:1299--1318, 2001.

\bibitem[FKYB99]{Felsenstein99}
J.~Felsenstein, M.K. Kuhner, J.~Yamamoto, and P.~Beerli.
\newblock Likelihoods on coalescents: a {Monte Carlo} sampling approach to
  inferring parameters from population samples of molecular data.
\newblock {\em IMS Lect. Notes Monogr. Ser.}, 33:163--185, 1999.

\bibitem[GJS08]{Griffiths08}
R.C. Griffiths, P.A. Jenkins, and Y.S. Song.
\newblock Importance sampling and the two--locus model with subdivided
  population structure.
\newblock {\em Adv. in Appl. Probab.}, 40:473--500, 2008.

\bibitem[GM96]{Griffiths96}
R.C. Griffiths and P.~Marjoram.
\newblock Ancestral inference from samples of {DNA} sequences with
  recombination.
\newblock {\em J. Comput. Biol.}, 3:479--502, 1996.

\bibitem[GT94a]{Griffiths94a}
R.C. Griffiths and S.~Tavar\'{e}.
\newblock Ancestral inference in population genetics.
\newblock {\em Statist. Sci.}, 9:307--319, 1994.

\bibitem[GT94b]{Griffiths94b}
R.C. Griffiths and S.~Tavar\'{e}.
\newblock Sampling theory for neutral alleles in a varying environment.
\newblock {\em Phil. Trans. R. Soc. Lond. B}, 344:403--410, 1994.

\bibitem[GT94c]{Griffiths94c}
R.C. Griffiths and S.~Tavar\'{e}.
\newblock Simulating probability distributions in the coalescent.
\newblock {\em Theor. Popln Biol.}, 46:131--159, 1994.

\bibitem[GT99]{Griffiths99}
R.C. Griffiths and S.~Tavar\'{e}.
\newblock The ages of mutations in gene trees.
\newblock {\em Ann. Appl. Probab.}, 9:567--590, 1999.

\bibitem[GT08]{Gorur08}
D.~Gorur and Y.W. Teh.
\newblock An efficient sequential {Monte Carlo} algorithm for coalescent
  clustering.
\newblock {\em NIPS}, 2008.

\bibitem[HUW08]{Hobolth08}
A.~Hobolth, M.~Uyenoyama, and C.~Wiuf.
\newblock Importance sampling for the infinite sites model.
\newblock {\em Stat. Appl. Genet. Mol.}, 7:Article 32, 2008.

\bibitem[Jen12]{Jenkins12}
P.A. Jenkins.
\newblock Stopping--time resampling and population genetic inference under
  coalescent models.
\newblock {\em Stat. Appl. Genet. Mol. Biol.}, 11(1):Article 9, 2012.

\bibitem[JG11]{Jenkins11}
P.A. Jenkins and R.C. Griffiths.
\newblock Inference from samples of {DNA} sequences using a two--locus model.
\newblock {\em J. Comput. Biol.}, 18:109--127, 2011.

\bibitem[Kin82]{Kingman82}
J.F.C. Kingman.
\newblock The coalescent.
\newblock {\em Stochast. Process. Appllic.}, 13(3):235--248, 1982.

\bibitem[LS03]{Li03}
N.~Li and M.~Stephens.
\newblock Modeling linkage disequilibrium and identifying recombination
  hotspots using single--nucleotide polymorphism data.
\newblock {\em Genetics}, 165:2213--2233, 2003.

\bibitem[M\"06]{Moehle06}
M.~M\"{o}hle.
\newblock On sampling distributions for coalescent processes with simultaneous
  multiple collisions.
\newblock {\em Bernoulli}, 12(1):35--53, 2006.

\bibitem[MS01]{Moehle01}
M.~M\"{o}hle and S.~Sagitov.
\newblock A classification of coalescent processes for haploid exchangeable
  population models.
\newblock {\em Ann. Probab.}, 29(4):1547--1562, 2001.

\bibitem[MS03]{Moehle03}
M.~M\"ohle and S.~Sagitov.
\newblock Coalescent patterns in exchangeable diploid population models.
\newblock {\em J. Math. Biol.}, 47:337--352, 2003.

\bibitem[MW96]{Meng96}
X.L. Meng and W.H. Wong.
\newblock Simulating rations of normalizing constants via a simple identity: a
  theoretical exploration.
\newblock {\em Statist. Sinica}, 6:831--860, 1996.

\bibitem[Pit99]{Pitman99}
J.~Pitman.
\newblock Coalescents with multiple collisions.
\newblock {\em Ann. Probab.}, 27(4):1870--1902, 1999.

\bibitem[PS10]{Paul10}
J.S. Paul and Y.S. Song.
\newblock A principled approach to deriving approximate conditional sampling
  distributions in population genetic models with recombination.
\newblock {\em Genetics}, 186:321--338, 2010.

\bibitem[PSS11]{Paul11}
J.S. Paul, M.~Steinr\"{u}cken, and Y.S. Song.
\newblock An accurate sequentially {M}arkov conditional sampling distribution
  for the coalescent with recombination.
\newblock {\em Genetics}, 187:1115--1128, 2011.

\bibitem[Sag99]{Sagitov99}
S.~Sagitov.
\newblock The general coalescent with asynchronous mergers of ancestral
  lineages.
\newblock {\em J. Appl. Probab.}, 36(4):1116--1125, 1999.

\bibitem[SBB13]{Steinrucken13a}
M.~Steinr\"{u}cken, M.~Birkner, and J.~Blath.
\newblock Analysis of {DNA} sequence variation within marine species using
  {B}eta--coalescents.
\newblock {\em Theor. Popln Biol.}, 87:15--24, 2013.

\bibitem[Sch00]{Schweinsberg00}
J.~Schweinsberg.
\newblock Coalescents with simultaneous multiple collisions.
\newblock {\em Electron. J. Probab.}, 5:1--50, 2000.

\bibitem[Sch03]{Schweinsberg03}
J.~Schweinsberg.
\newblock Coalescent processes obtained from super--critical {G}alton--{W}atson
  processes.
\newblock {\em Stoch. Proc. Appl.}, 106:107--139, 2003.

\bibitem[SD00]{Stephens00}
M.~Stephens and P.~Donnelly.
\newblock Inference in molecular population genetics.
\newblock {\em J. R. Statist. Soc. B}, 62(4):605--655, 2000.

\bibitem[SHS13]{Sheehan13}
S.~Sheehan, K.~Harris, and Y.S. Song.
\newblock Estimating variable effective population sizes from multiple genomes:
  A sequentially markov conditional sampling distribution approach.
\newblock {\em Genetics}, 194:647--662, 2013.

\bibitem[SPS13]{Steinrucken13b}
M.~Steinr\"{u}cken, J.S. Paul, and Y.S. Song.
\newblock A sequentially markov conditional sampling distribution for
  structured populations with migration and recombination.
\newblock {\em Theor. Popln Biol.}, 87:51--61, 2013.

\bibitem[SW08]{Sargsyan08}
O.~Sargsyan and J.~Wakeley.
\newblock A coalescent process with simultaneous multiple mergers for
  approximating the gene genealogies of many marine organisms.
\newblock {\em Theor. Popln Biol.}, 7:104--114, 2008.

\bibitem[TV09]{Taylor09}
J.E. Taylor and A.~V\'{e}ber.
\newblock Coalescent processes in subdivided populations subject to recurrent
  mass extinctions.
\newblock {\em Electron. J. Probab.}, 14:242--288, 2009.

\end{thebibliography}

%
%
%
%

\end{document}